\documentclass[12pt,oneside]{amsart}

\usepackage{amssymb,latexsym}

\usepackage{hyperref}
\hypersetup{
  colorlinks   = true, 
  urlcolor     = blue, 
  linkcolor    = blue, 
  citecolor   = red 
}
\usepackage{anysize}

\usepackage{pgfplots}
\usepackage{mathrsfs}
\usetikzlibrary{arrows}
\usetikzlibrary{calc}
\usetikzlibrary{intersections}
\usetikzlibrary{through}

\usepackage{mathtools}

\newtheorem{theorem}{Theorem}[section]

\newtheorem{lemma}[theorem]{Lemma}

\theoremstyle{definition}
\newtheorem{definition}[theorem]{Definition}

\theoremstyle{remark}
\newtheorem{remark}[theorem]{Remark}

\numberwithin{equation}{section}

\DeclareMathOperator{\vol}{vol}
\DeclareMathOperator{\area}{area}
\DeclareMathOperator{\aff}{aff}
\DeclareMathOperator{\conv}{conv}

\newcommand{\const}{{\rm const}}

\renewcommand{\epsilon}{\varepsilon}
\renewcommand{\phi}{\varphi}
\renewcommand{\kappa}{\varkappa}

\begin{document}

\title{Rogers's proof of Vaaler's theorem}

\author{Roman Karasev}


\thanks{The research was carried out within the state assignment 1.1.1-0029/25 of Ministry of Science and Higher Education of the Russian Federation for IITP RAS} 



\address{Roman Karasev, Institute for Information Transmission Problems RAS, Bolshoy Karetny per. 19, Moscow, Russia 127994 and Moscow Institute of Physics and Technology, Institutskiy per. 9, Dolgoprudny, Russia 141700}
\email{r\_n\_karasev@mail.ru}
\urladdr{http://www.rkarasev.ru/en/}



\begin{abstract}
We note that an argument by Rogers (1958) gives a proof of Vaaler's theorem (1979) about sections of the cube and allows certain generalizations of the theorem.
\end{abstract}

\maketitle

\section{Introduction}

Vaaler's theorem~\cite{vaaler1979} asserts that \emph{any section of the cube $[-1,1]^N$ by a linear $n$-dimensional subspace has $n$-dimensional volume at least $2^n$.} Different proofs of this theorem are known. The original approach was by defining the ``more peaked'' partial order on densities and developing its properties, but other approaches are also able to prove it, see~\cite[Section~2.2]{ahk2016} for example. 

Here we show that a relatively straightforward approach to Vaaler's theorem was found by Rogers in~\cite{rogers1958} more than twenty years before the publication of Vaaler's proof. In fact, Rogers's approach allows to prove certain generalizations that we state below.

\begin{theorem}
\label{theorem:rogers}
Let $P\subset\mathbb R^n$ be a polyhedron containing the origin in its interior and having the property that for any face $F\subset P$ of codimension $k$ the distance from $0$ to the affine span of $F$ is at least $\sqrt k$. Then the volume of $P$ is at least $2^n$.
\end{theorem}

The equality in the above theorem is attained when $P$ is a unit cube itself. The next theorem is a modification of Vaaler's theorem that was conjectured (for sections of the cube) by Grigory~M.~Ivanov in private communication. A version of this question for $n=2$ was discussed on \href{https://mathoverflow.net/questions/264805/shortest-path-connecting-two-opposite-points-on-a-cube}{mathoverflow.net} and an upper bound in codimension $1$ was established in~\cite{konig-koldobsky2019}. We only establish the following result in dimensions $n\le 3$.

\begin{theorem}
\label{theorem:rogers-surface}
Let $P\subset\mathbb R^n$ for $n=2,3$ be a polyhedron containing the origin in its interior and having the property that for any face $F\subset P$ of codimension $k$ the distance from $0$ to the affine span of $F$ is at least $\sqrt k$. Then the $(n-1)$-dimensional volume of $\partial P$ $($the surface area$)$ is at least $n 2^n$.
\end{theorem}

Again, the equality in this theorem is attained when $P$ is a unit cube. See the classical book~\cite{zong2009} for discussions of other results about the sections of the cube, and the paper~\cite{ivanov2020} for another elementary approach to extremal sections and projections, as well as many open questions in the area.

\subsection*{Acknowledgments} The author thanks Grigory M. Ivanov for the discussions and suggestions, and the anonymous referees for carefully reading the text and suggesting several corrections.

\section{Proof of the volume estimate}

We start with a quick explanation of the relation of the above statements to Vaaler's theorem.

\begin{proof}[Deduction of Vaaler's theorem from Theorem~\ref{theorem:rogers}]
The cube $[-1,1]^N$ evidently has the property that for any face $F\subset[-1,1]^N$ of codimension $k$ the distance from $0$ to the affine span of $F$ is at least $\sqrt k$. It remains to note that this property preserves on passing to the section $P=[-1,1]^N\cap L$ since any codimension $k$ face of $P$ is contained in a codimension $k$ face of $[-1,1]^N$.
\end{proof}

\begin{proof}[Proof of Theorem~\ref{theorem:rogers}, following the proof in~\cite{rogers1958}]
For any face $F\subseteq P$ take the point $a_F\in F$ that is closest to the origin. For any chain of faces
\[
P = F_0 \supset F_1 \supset \dots \supset F_n
\]
(the index denotes the codimension) consider the simplex $a_0a_1\dots a_n$, setting $a_k = a_{F_k}$ for brevity and observe (by induction on the skeleta of $F$) that these simplices do not overlap and cover $P$. The situation is similar to a barycentric subdivision, but here a point $a_F$ need not lie in the relative interior of $F$, hence some simplices of the subdivision may be degenerate.

Now we consider one such non-degenerate simplex $A=\conv\{a_0,a_1,\ldots, a_n\}$ and also consider another simplex $B=\conv\{b_0,b_1,\ldots, b_n\}$, where $b_k$ is the point of the affine span $\aff F_k$ closest to the origin. By the assumption of the theorem $|b_k|\ge \sqrt{k}$ for any $k$. The simplex $B$ has the property that its edges $b_kb_{k+1}$ are pairwise orthogonal, which is easy to see in the orthogonal coordinate system in which the flag of affine subspaces $\aff F_k$ is the flag of coordinate subspaces.

Let us show that the affine transformation of $A$ to $B$ taking every $a_k$ to $b_k$ does not increase $|x|$ of any point $x\in A$. This affine transformation can be made in $n$ steps, where on step $k$ one transforms the simplex $\conv\{b_0,\ldots,b_{k-1}, a_k, \ldots, a_n\}$ to $\conv\{b_0,\ldots,b_{k-1}, b_k, a_{k+1}\ldots, a_n\}$. On this transformation step the facet $\conv\{b_0,\ldots,b_{k-1}, a_{k+1}, \ldots, a_n\}$ is fixed, and the vertex $a_k$ moves to $b_k$. We assume that $a_k\neq b_k$, otherwise nothing happens on this transformation step.

Note that $a_k$ is the point of $F_k$ closest to the origin, and also closest to $b_k$ from the orthogonality property of the edges of $B$ mentioned above. The hyperplane $H$ through $b_k$ orthogonal to $a_kb_k$ has $F_k$ on one of its sides. From the orthogonality property of $B$ the affine span of $b_0,\ldots, b_k$ is orthogonal to the affine span of $F_k$ and is therefore contained in $H$. Since $a_m\in F_k$ for any $m\ge k$, $H$ has $\conv\{b_0,\ldots,b_{k-1}, a_k, \ldots, a_n\}$ and $\conv\{b_0,\ldots,b_{k-1}, b_k, a_{k+1}\ldots, a_n\}$ on one of its sides. Since under the affine transformation step that we consider every point $x$ of the simplex  $\conv\{b_0,\ldots,b_{k-1}, a_k, \ldots, a_n\}$ moves in the direction parallel to $b_k-a_k$ and orthogonal to $H$, the distance from this point to $0\in H$ cannot increase.

Let $C$ be the simplex with vertices $c_i$ defined by
\[
c_0 = 0,\quad c_k=(\underbrace{1,\ldots,1}_k , 0, \ldots, 0).
\]
Another affine transformation $B\to C$ taking every $b_k$ to its respective $c_k$ does not increase the distance to the origin by Lemma~\ref{lemma:orthoscheme-volume} below. The resulting affine transformation $f : A\to C$ also does not increase the distance to the origin and therefore has the property for any $r>0$ for intersection with the ball $B_0(r)$ we have
\[
f(B_0(r)\cap A)\subseteq B_0(r)\cap C,
\]
which allows to conclude
\[
\frac{\vol B_0(r)\cap A}{\vol A} = \frac{\vol f(B_0(r)\cap A)}{\vol C} \le \frac{\vol B_0(r)\cap C}{\vol C}.
\]
The simplex $C$ is in fact congruent to one of the $2^n n!$ simplices of the barycentric subdivision of $[-1,1]^n$, hence for $r=1$ we have
\[
\frac{\vol B_0(1)\cap C}{\vol C} = \frac{\vol B_0(1)}{\vol [-1,1]^n} = \frac{\vol B_0(1)}{2^n}.
\]
Summing the thus obtained inequalities
\[
\vol A \ge \frac{2^n}{\vol B_0(1)} \vol B_0(1)\cap A
\]
over all $A$ of the subdivision of $P$ in view of $B_0(1)\subset P$ (that follows from the assumption of the theorem) we obtain
\[
\vol P \ge \frac{2^n}{\vol B_0(1)} \vol B_0(1)\cap P = \frac{2^n}{\vol B_0(1)} \vol B_0(1) = 2^n.
\]
\end{proof}

Let us use the standard terminology for simplices like $B$ in the above proof.

\begin{definition}
A simplex $\conv\{p_0,\ldots,p_n\}\subseteq \mathbb R^n$ is an \emph{orthoscheme} if $\conv\{p_0,\ldots, p_k\}$ is orthogonal to $\conv\{p_k,\ldots, p_n\}$ for every $k=1,\ldots, n-1$.
\end{definition}

\begin{lemma}
\label{lemma:orthoscheme-volume} Let $B=\conv\{b_0,\ldots, b_n\}$ and $C=\conv\{c_0,\ldots, c_n\}$ be two orthoschemes such that $b_0=c_0=0$ and $|b_k|\ge |c_k|$ for all $k$. Then the affine transformation $f : B\to C$ defined by interpolating $f(b_k)=c_k$ has the property that, for any $x\in B$, $|f(x)|\le |x|$. In particular, for any $r>0$ 
\[
\frac{\vol B}{\vol B_0(r)\cap B} \ge \frac{\vol C}{\vol B_0(r)\cap C}.
\]
\end{lemma}
\begin{proof}
From the orthoscheme property of $B$ one may assume that 
\[
b_0 = 0,\quad b_k = (\beta_1,\ldots,\beta_k, 0, \ldots, 0)
\]
in an appropriately chosen orthonormal coordinate system. For $C$ one may also assume that
\[
c_0 = 0,\quad c_k=(\gamma_1,\ldots,\gamma_k, 0, \ldots, 0)
\]
in some other orthonormal coordinate system.

Set $\lambda_k = \gamma_k/\beta_k > 0$. The map $f : B\to C$ in these coordinates is the linear transformation by a diagonal matrix with $\lambda_1,\ldots, \lambda_n > 0$ on the diagonal. Let us show that under this diagonal transformation $|x|$ cannot increase for any $x\in B$. A point 
\[
x = (\beta_1 x_1, \ldots, \beta_m x_n)
\]
is in $B$ if and only if
\[
x_1\ge x_2\ge \dots \ge x_n\ge 0.
\]
Its image is
\[
f(x) = (\gamma_1 x_1, \ldots, \gamma_m x_n)
\]
and the inequality $|f(x)|^2\le |x|^2$ in coordinates reads
\[
\beta_1^2 x_1^2 + \dots + \beta_n^2 x_n^2 \le \gamma_1^2 x_1^2 + \dots + \gamma_n^2 x_n^2.
\]
This is a linear inequality on the vector $(x_1^2,x_2^2, \ldots, x_n^2)$ that satisfies the system of linear inequalities following from those defining $B$:
\[
1\ge x_1^2\ge x_2^2\ge \dots \ge x_n^2\ge 0.
\]
Hence it only remains the check the inequality for the vertices of the domain polyhedron:
\[
(x_1^2,\ldots, x_n^2) = (\underbrace{1, \ldots, 1}_k,\underbrace{0,\ldots,0}_{n-k}),
\]
that is, to check the inequality $|c_k|^2 \le |b_k|^2$. The latter holds true by the assumption of the lemma.
\end{proof}

\begin{remark}
The version of Theorem~\ref{theorem:rogers} in~\cite{rogers1958} used a different assumption on distances: For any face $F\subset P$ of codimension $k$ the distance from $0$ to the affine span of $F$ is at least $\sqrt \frac{2k}{k+1}$. From Lemma~\ref{lemma:nplus1obtuse} (see below) one infers that this assumption is satisfied for $P$ a Voronoi cell in a packing of balls in $\mathbb R^n$ of unit radius (a \emph{Voronoi cell} of a ball is the set of points in $\mathbb R^n$ non-strictly closer to the center of this ball than to the centers of any other ball in the packing). After that the argument is basically the same as in the proof of Theorem~\ref{theorem:rogers}, but the resulting simplex $C$ is this time congruent to any simplex of the barycentric subdivision of a regular simplex $T\subset\mathbb R^n$ of edge length $2$. This results in what is called above \emph{the  simplex upper bound} for the packing density.
\end{remark}

\begin{lemma}[Given only to complete the above remark]
\label{lemma:nplus1obtuse}
For a system of $k+1$ unit vectors $u_0,\ldots u_k\in\mathbb R^n$ for some $i\neq j$ the inequality $u_i\cdot u_j\ge -1/k$ holds, that is, the angle between $u_i$ and $u_j$ is at most $\arccos\frac{-1}{k}$.
\end{lemma}
\begin{proof}
Assume the contrary, $u_i\cdot u_j < -1/k$ for any $i\neq j$ and write
\[
(u_0+\dots + u_k)\cdot (u_0+\dots + u_k) = \sum_i u_i^2 + \sum_{i \neq j} u_i\cdot u_j < k+1 - k(k+1)\frac{1}{k} = 0,
\]
which is a contradiction.
\end{proof}
\section{Proof of the surface area estimate}

\begin{proof}[Proof of Theorem~\ref{theorem:rogers-surface}]
We follow the proof of Theorem~\ref{theorem:rogers} above with appropriate modifications and pass to lower bounds for a single simplex. The simplex $A$ was affinely transformed to $C$ in two stages $A\to B\to C$, and it was established that
\[
\frac{\vol A}{\vol B_0(1)\cap A} \ge \frac{\vol B}{\vol B_0(1)\cap B} \ge \frac{\vol C}{\vol B_0(1)\cap C}.
\]
If one wants to bound the surface area of $P$ (the $(n-1)$-dimensional Riemannian volume of $\partial P$) from below then one may observe that $A$ and $B$ have the same distance from the origin to the affine spans of opposite to the origin facets, hence
\[
\frac{\vol_{n-1} \conv \{a_1,\ldots, a_n\} }{\vol A} = \frac{\vol_{n-1} \conv \{b_1,\ldots, b_n\}}{\vol B}
\]
and therefore
\[
\frac{\vol_{n-1} \conv \{a_1,\ldots, a_n\} }{\vol B_0(1)\cap A} \ge \frac{\vol_{n-1} \conv \{b_1,\ldots, b_n\}}{\vol B_0(1)\cap B}.
\]
In order to prove that the minimal $\vol_{n-1}\partial P$ is attained for the $n$-dimensional cube with edge length $2$ composed of copies of the simplex $C$, it remains to show that 
\[
\frac{\vol_{n-1} \conv \{b_1,\ldots, b_n\} }{\vol B_0(1)\cap B} \ge \frac{\vol_{n-1} \conv \{c_1,\ldots, c_n\}}{\vol B_0(1)\cap C}.
\]
It is sufficient to prove the inequality 
\begin{equation}
\label{equation:circle-move}
\frac{\vol_{n-1} \conv \{b_1,\ldots, b_n\} }{\vol B_0(1)\cap B} \ge \frac{\vol_{n-1} \conv \{b'_1,\ldots, b'_n\}}{\vol B_0(1)\cap B'}
\end{equation}
for the transformation $B\to B'$ such that $b'_i=b_i$ for all $i\neq 1$ and $b'_1$ is such that $|b'_1|=1$ and $b_0b'_1$ is orthogonal to $\conv\{b'_1,b_2,\ldots, b_n\}$. Under these assumptions $b'_1$ is in the affine plane spanned by $b_0,b_1,b_2$ and lies on the same circle with diameter $b_0b_2$ in this plane. Parameterize $b_1(t)$ by the angle $t = \angle b_2b_0b_1(t)$, set $B(t)=\conv \{b_0, b_1(t), b_2, \ldots b_n\}$ and notice that the numerator of
\[
\frac{\vol_{n-1} \conv \{b_1(t),\ldots, b_n\} }{\vol B_0(1)\cap B(t)}
\]
is proportional to $|b_1(t)b_2| = |b_0b_2| \sin t$ from the orthogonality of $b_1(t)b_2$ to $\conv\{b_2,\ldots, b_n\}$, while the denominator is the solid angle at the vertex $b_0$ of the simplex $B'(t)$. For dimension $n=2$ \eqref{equation:circle-move} follows from decreasing of $\frac{\sin t}{t}$ in $t\in (0,\pi)$. For dimension $n=3$ \eqref{equation:circle-move} follows from Lemma~\ref{lemma:spherical-triangle} on solid angles in dimension $3$ (spherical triangles) below.

The remaining transformation $B'\to C$ will then keep the unit distance from the origin to the opposite facet of $B'$ and $C$ and will therefore satisfy 
\[
\frac{\vol_{n-1} \conv \{b'_1,\ldots, b'_n\} }{\vol B_0(1)\cap B'} \ge \frac{\vol_{n-1} \conv \{c_1,\ldots, c_n\}}{\vol B_0(1)\cap C}
\]
by the same argument as for the transformation $A\to B$ above and Lemma~\ref{lemma:orthoscheme-volume}. Eventually, we obtain a lower bound that must be attained for $P=[-1,1]^n$ barycentrically subdivided into simplices congruent to $C$ and observe that $\vol_{n-1}\partial [-1,1]^n = n2^n$.
\end{proof}

\begin{lemma}
\label{lemma:spherical-triangle}
Let $T(t)=s_1(t)s_2s_3$ be a spherical triangle with $|s_1s_2|=t < \pi/2$, $|s_2s_3|=\const <\pi/2$, and $\angle s_1s_2s_3 = \pi/2$. Then 
\[
\frac{\area T(t)}{\sin t}
\]
is increasing in $t$.
\end{lemma}
\begin{proof}
Let us make a central projection from the sphere to the plane tangent to the sphere at $s_2$. The triangle is then
\[
T(t) = \{(x,y)\in \mathbb R^2\ |\ x,y\ge 0, x/\tan t + y/q \le 1\}.
\]
Then
\begin{multline*}
\area T(t) = \int_{T(t)} \frac{dxdy}{(1+x^2+y^2)^{3/2}} 
= \int_0^q \int_0^{(1 - y/q)\tan t} \frac{dx}{(1+x^2+y^2)^{3/2}}\; dy = \\
= \int_0^q \frac{(1 - y/q)\tan t}{\sqrt{1+y^2+((1 - y/q)\tan t)^2}(1+y^2)}\; dy.
\end{multline*}
Then
\[
\frac{\area T(t)}{\sin t} = \frac{\area T(t) \sqrt{1+ \tan^2 t}}{\tan t} 
= \int_0^q \sqrt{\frac{1+ \tan^2 t}{1+y^2+((1 - y/q)\tan t)^2}} \frac{1 - y/q}{(1+y^2)}\; dy.
\]
The expression under the integral is increasing in $t$ since every 
\[
\frac{1+u}{a+bu}
\]
is increasing in $u=\tan^2 t>0$ when $a>b>0$, as is easily seen after taking derivative.
\end{proof}

\bibliography{../Bib/karasev}

\begin{thebibliography}{1}

\bibitem{ahk2016}
A.~Akopyan, A.~Hubard, and R.~Karasev.
\newblock Lower and upper bounds for the waists of different spaces.
\newblock {\em Topological Methods in Nonlinear Analysis}, 53(2):457--490,
  2019.
\newblock \href{https://arxiv.org/abs/1612.06926}{arXiv:1612.06926}.

\bibitem{ivanov2020}
G.~M. Ivanov.
\newblock Tight frames and related geometric problems.
\newblock {\em Canadian Mathematical Bulletin}, 64(4):942--963, 2020.

\bibitem{konig-koldobsky2019}
H.~K\"onig and A.~Koldobsky.
\newblock On the maximal perimeter of sections of the cube.
\newblock 346:773--804, 2019.

\bibitem{rogers1958}
C.~A. Rogers.
\newblock The packing of equal spheres.
\newblock {\em Proceedings of the London Mathematical Society}, 3(8):609--620,
  1958.

\bibitem{vaaler1979}
J.~D. Vaaler.
\newblock A geometric inequality with applications to linear forms.
\newblock {\em Pacific Journal of Mathematics}, 83(2):543--553, 1979.

\bibitem{zong2009}
C.~Zong.
\newblock {\em The Cube---A Window to Convex and Discrete Geometry}.
\newblock Cambridge University Press, 2009.

\end{thebibliography}
\bibliographystyle{abbrv}
\end{document}